\documentclass{amsart}

\usepackage{graphicx}      
\usepackage{natbib}        

\usepackage{float} 
\usepackage{amssymb} 
\usepackage{curves} 
\usepackage{xcolor} 
\usepackage{eepic} 
\usepackage{enumitem} 
\usepackage{fancyvrb} 
\usepackage{breqn} 
\usepackage{csquotes} 
\usepackage{tikz} 
\usetikzlibrary{patterns} 
\usepackage{tabstackengine} 
\usepackage{tabularx} 
\stackMath 
\usepackage{algorithm,algorithmic}
\usepackage{amsaddr}

\RequirePackage{color}
\RequirePackage[colorlinks, urlcolor=my-blue,linkcolor=my-red,citecolor=my-green]{hyperref}
\definecolor{my-blue}{rgb}{0.0,0.0,0.6}
\definecolor{my-red}{rgb}{0.5,0.0,0.0}
\definecolor{my-green}{rgb}{0.0,0.5,0.0}
\definecolor{nicos-red}{rgb}{0.75,0.0,0.0}
\definecolor{light-gray}{gray}{0.6}
\definecolor{really-light-gray}{gray}{0.8}
\definecolor{sussexg}{rgb}{0.0,0.5,0.5}
\definecolor{sussexp}{rgb}{0.5,0.0,0.5}

\definecolor{darkgreen}{rgb}{0.0,0.5,0.0}
\definecolor{darkblue}{rgb}{0.0,0.0,0.3}
\definecolor{nicosred}{rgb}{0.65,0.1,0.1}
\definecolor{light-gray}{gray}{0.7}
\allowdisplaybreaks[1]

\DeclareMathOperator{\diag}{diag}

\newcommand{\cP}{{\mathcal{P}}}

\newcommand{\cK}{{\mathcal{K}}}

\newcommand{\bR}{{\mathbb{R}}}
\newcommand{\bZ}{{\mathbb{Z}}}
\newcommand{\bN}{{\mathbb{N}}}
\newcommand{\OP}{\Omega^{P}}
\newcommand{\OPtt}{\Omega^{P^{\theta}}}
\newcommand{\xltt}{\mathbf{x}^{\ell,\theta}}

\newcommand{\Ntt}{N^{\theta}}

\newcommand{\Nee}{N^{\eta}}
\newcommand{\biltt}{b_{i}^{\ell,\theta}}

\newcommand{\Rltt}{R^{\ell,\theta}}
\newcommand{\lliltt}{\lambda_{i}^{\ell,\theta}}

\newcommand{\llssltt}{\lambda_{\sigma}^{\ell,\theta}}
\newcommand{\aassiltt}{\alpha_{\sigma,i}^{\ell,\theta}}
\newcommand{\bbiltt}{\beta_{i}^{\ell,\theta}}

\newcommand{\rrltt}{\rho^{\ell,\theta}}
\newcommand{\Ptt}{P^{\theta}}
\newcommand{\cPtt}{{\mathcal{P}}^{\theta}}
\newcommand{\pitt}{p_{i}^{\theta}}
\newcommand{\riltt}{r_{i}^{\ell,\theta}}
\newcommand{\Pltt}{P^{\ell,\theta}}
\newcommand{\piltt}{p_{i}^{\ell,\theta}}
\newcommand{\aijtteelm}{a_{i,j}^{\theta,\eta;\ell,m}}
\newcommand{\xeett}{\mathbf{x}^{m,\eta}}
\newcommand{\pjmee}{p_{j}^{m,\eta}}
\newcommand{\lljmee}{\lambda_{j}^{m,\eta}}
\newcommand{\rrmee}{\rho^{m,\eta}}

\newtheorem{theorem}{Theorem}[section]
\newtheorem{lemma}[theorem]{Lemma}

{\theoremstyle{definition}
\newtheorem{defn}[theorem]{Definition}
\newtheorem{example}[theorem]{Example}

}

\let\emptyset\varnothing
 
\begin{document}

\title[Lyapunov functions for switched systems]{Construction of Lyapunov Functions for Switched Systems using Meshfree Collocation} 

\newcommand\blfootnote[1]{%
  \begingroup
  \renewcommand\thefootnote{}\footnote{#1}%
  \addtocounter{footnote}{-1}%
  \endgroup
}

\blfootnote{This research was supported by EPSRC PhD studentship funding.}

\author{Jay Ward$^{1}$, Nicos Georgiou$^1$ and Peter Giesl$^1$}
\address{$^1$University of Sussex, United Kingdom}
\email{Jay.Ward@sussex.ac.uk, N.Georgiou@sussex.ac.uk and P.A.Giesl@sussex.ac.uk}

\begin{abstract}                
Switched systems are a family of dynamical systems where a switching rule indicates which system is \enquote{switched on}. This rule can be dependent on time and/or position in the state space. Stability of switched systems is a property that is often investigated using the existence of  one (or multiple) Lyapunov function(s). 
We develop an algorithm using a scattered approximation method to construct a Lyapunov function for switched systems, accompanied with stability results. The construction adapts a previous method for autonomous ODEs that uses meshfree collocation and quadratic programming. 
\end{abstract}

\maketitle


\section{Introduction}

Lyapunov functions characterise the behaviour of a dynamical system. Expressing a dynamical system as a system of differential equations, a (strict) Lyapunov function is a function that decreases along solutions of this system. For a smooth function this is equivalent to it having a negative orbital derivative. Existence of a Lyapunov function implies stability and attractivity of an equilibrium and, moreover, determines the basin of attraction through sublevel sets. For relevant definitions and results see \cite{hahn1967stability}, \cite{zubov1961methods} and \cite{KhalilNonlinear}.\par
It is therefore a useful endeavour to investigate methods of obtaining a Lyapunov function. \cite{giesl2007meshless} introduced a meshfree collocation method to construct Lyapunov functions for autonomous ODEs with equilibrium at the origin by approximating a solution to the PDE $DV(\mathbf{x})=-\|\mathbf{x}\|^2$; $D$ represents the orbital derivative. In this method, a finitely many collocation points are selected and the goal is to find the norm-minimal function such that the PDE is satisfied at these points. The advantages of this method are numerous (\cite{giesl2007constructionRB}, \cite{giesl2018construction}): scattered points can be added, no triangulation of the state space is needed, the approximating function is smooth, the method works in any dimension and the resulting interpolation problem is uniquely solvable when expressed in terms of radial basis functions. \cite{giesl2018construction} extended this approach using quadratic programming to find an approximation that satisfies a system of partial differential inequalities, i.e. $DV(\mathbf{x})\leq{0}$.\par

We further generalize the method to address switched systems, an important class of dynamical systems where discrete switching events occur between continuous-time systems. These systems arise in a wide range of applications in sciences, engineering and economics, e.g. modeling cellular processes, air traffic control transitions, or thermostat regulation. We define a construction method for a Lyapunov function of a switched system with time- and state-dependent switching using meshfree collocation and quadratic programming.\par
Section \ref{secSwitched} reviews types of switched systems and existing construction methods. Section \ref{secConstruction} defines a switched system and its Lyapunov function, and presents our construction method, summarised in Algorithm \ref{algMain}. Theorem \ref{thmStability} links Lyapunov functions to asymptotic stability, while Theorem \ref{lemMain} confirms that our function satisfies the Lyapunov conditions. Section \ref{secEx} shows examples using Algorithm \ref{algMain} for three systems.

\section{Switched systems}
\label{secSwitched}

A switched system is a family of dynamical systems where the system that describes the dynamics changes with time and/or the position in the state space. 
A system where time-dependent uncontrolled switching takes place is referred to as an arbitrary switched system, and a system with state-dependent switching is referred to as a variable structure system. An overview of both can be found in \cite{liberzon2003switching}. Results for arbitrary switched systems can be found in \cite{hafstein2009algorithm}, where Lyapunov's direct method was generalised.\par
We consider switched systems that allow for both time- and state-dependent switching. Systems will be defined on subsets of the state space, with the union of these subsets making up the whole state space. 
The subsets may overlap; wherever this occurs, the system may switch arbitrarily to a different evolution.
In \cite{liberzon2003switching}, stability conditions for such a system are described. It states that these conditions are achieved if a Lyapunov function exists for each system in the subset where it is defined. This means that multiple Lyapunov functions can be considered. The method used to prove stability for a similar switched system found in \cite{hou1996stability} also uses multiple Lyapunov functions; solution trajectories are described as a sequence of switching times paired with the index of the system that is being switched to. The switched system is then stable if there exist functions that decrease wherever their respective system is \enquote{switched on}.\par
The method in \cite{hafstein2009algorithm} constructs Lyapunov functions for arbitrary switched systems, with an extension to a variable structure example. In \cite{hafstein2022sliding}, differential inclusions are used to construct Lyapunov functions for variable structure systems.

\section{Construction using Quadratic programming}
\label{secConstruction}

Let $\Omega\subset\bR^{d}$ be a bounded domain with Lipschitz boundary and define a switched system by autonomous systems of differential equations
\begin{equation}
    \label{eqnfis}
    \dot{\mathbf{x}}(t)=\mathbf{f}_{i}(\mathbf{x}(t))\text{, }i\in{I}:=\{1,\dots,k\}
\end{equation}
where $\mathbf{x}(t)\in\Omega\subset\bR^{d}$. We assume that the functions $\mathbf{f}_{i}$ defined on the domain $C_{i}\subset\Omega$, $\mathbf{0}\in{C_{i}}$ are locally Lipschitz continuous and have an equilibrium point at the origin. We require that $C_{i}=\overline{(C_{i})^{\circ}}$ and that the sets $C_{i}$ for all $i\in{I}$ are a covering of $\Omega$. In this paper we consider systems such that sliding modes are avoided, more details on these can be found in \cite{hafstein2022sliding}. A sliding mode can be interpreted as infinitely fast switching (\cite{liberzon2003switching}), so we require that for our systems there will be finite switches in finite time.\par
Switching in the system (\ref{eqnfis}) is when the index $i$ changes. We will use a definition of a switched system which allows both arbitrary and variable structure switching to take place, based on the domains $C_{i}$ of the functions $\mathbf{f}_{i}$. In any intersections of the domains (see (\ref{eqnOPDef})) $i$ can change arbitrarily between the values for which $\mathbf{f}_{i}$ is defined. If the boundary of one of these intersections is reached, $i$ might need to change to reflect the functions defined in this new subset. To illustrate this, consider Figure \ref{fig:intersectExample} where two functions $\mathbf{f}_{1}$, $\mathbf{f}_{2}$ are defined on the domains $C_{1}$, $C_{2}$ respectively. In the subset $C_{1}\cap{C}_{2}$ the system can switch arbitrarily between $i=1,2$, and if for example the boundary of the subset $C_{1}\cap{C}_{2}^{c}$ is reached, then the system will switch to $i=1$. Our algorithm will be applied to arbitrary switched, variable structure and combined systems in Examples \ref{exArb}, \ref{exVarStruct} and \ref{exCombo}.\par
A Lyapunov function of the switched system (\ref{eqnfis}) is a function that is non-increasing along the solutions of each differential equation in the system. For our purposes we say that $V:\Omega\rightarrow\bR$, $V(\mathbf{0})=0$ is a Lyapunov function for the system (\ref{eqnfis}) if and only if for all $i\in{I}$
\begin{equation}
    \label{eqnLyapSwitch}
    \nabla{V(\mathbf{x})}\cdot\mathbf{f}_{i}(\mathbf{x})\leq-\gamma(\|\mathbf{x}\|_{2})\text{ and }V(\mathbf{x})>0
\end{equation}
for all $\mathbf{x}\in{C_{i}}\setminus \{0\}$, where $\gamma$ is a function of class $\cK$, i.e. $\gamma\in{C([0,a),\bR^{+})}$ where $a>0$, $\gamma(0)=0$ and it is strictly monotonically increasing (\cite{KhalilNonlinear}). This $V$ is actually a strict Lyapunov function for the system (\ref{eqnfis}). Using this we can obtain the following stability result.
\begin{theorem}
    \label{thmStability}
    If there exists a function $V\in{C^{1}(\Omega)}$ that satisfies properties (\ref{eqnLyapSwitch}) for all $i\in{I}$, then the equilibrium $\mathbf{x}=\mathbf{0}$ of system (\ref{eqnfis}) is asymptotically stable.
\end{theorem}
This can be shown using methods similar to classical Lyapunov proofs that can be found in \cite{hahn1967stability}, \cite{zubov1961methods} and \cite{KhalilNonlinear}.\par
Now we introduce the algorithm that we will use to construct Lyapunov functions of our systems (\ref{eqnfis}); a summary of where the details of each step and definitions of any notation used can be found in the following sections is provided underneath.
\begin{algorithm}[H]
\caption{Lyapunov functions for Switched Systems}
\label{algMain}
    \begin{algorithmic}[1]
        \REQUIRE State space $\Omega$ and partition $\Ptt$ for $\theta=1,\dots,\Theta$ (see Section \ref{secPartition}), functions $\mathbf{f}_{i}(\mathbf{x})$ (\ref{eqnfis}), Wendland function $\Psi_0(r)$ (see Section \ref{secRKHS}) and evaluation points.
        \STATE Fix collocation points $X^{\theta}=\{\xltt\}_{1\leq\ell\leq\Ntt}$ in each set $\Ptt$, $\theta=1,\dots,\Theta$.
        \STATE If $0$ is selected as a collocation point remove it.
        \STATE For each collocation point $\xltt$ compute a basis of the vectors $\mathbf{f}_{i}(\xltt)$ and use this to find the values of the coefficients $\aassiltt$ (\ref{eqnllssltt}).
        \STATE Find $Q$ (\ref{eqnQ}).
        \STATE Find $A$ (\ref{eqnAelem}).
        \STATE Solve a quadratic programming problem for $\beta$ (\ref{eqnbTAb}).
        \STATE Find values of $v$ (\ref{eqnLyapv}) at evaluation points and plot. For all $i\in{I}$ find values of the orbital derivatives $\nabla{v}\cdot\mathbf{f}_{i}$ at evaluation points in $C_i$ and plot.
     \end{algorithmic}
\end{algorithm}
Section \ref{secMinimise} contains the definition of 
 the minimisation  problem (\ref{eqnMinProbv}) we are looking to solve in order to construct a Lyapunov function. Details on steps $3$ and $4$ are found in Section \ref{secBasis}, and on steps $5$, $6$ and $7$ in Section \ref{secQuadProg}, including  
 Theorem \ref{lemMain}, which verifies that our function $v$ (\ref{eqnLyapv}) is a solution of the minimisation problem (\ref{eqnMinProbv}).

\subsection{Partitioning the State Space}
\label{secPartition}

We define the following sets for ${P}\subseteq{I}$, $P\neq\emptyset$
\begin{equation}
    \label{eqnOPDef}
    \OP:=\bigg(\bigcap_{i\in{P}}C_{i}\bigg)\cap\bigg(\bigcap_{j\in{I\setminus{P}}}C_{j}^{c}\bigg).
\end{equation}
An example of this partition is shown in Figure \ref{fig:intersectExample}. Here we have two functions $\mathbf{f}_{1}$, $\mathbf{f}_{2}$ defined on the domains $C_{1}$, $C_{2}$ respectively. In this case, $I=\{1,2\}$ and the possible subsets $P$ are $\{1\}$, $\{2\}$ and $\{1,2\}$. $\OP$ consists of the three subsets $C_{1}\cap{C_{2}^{c}}$, $C_{1}\cap{C_{2}}$ and $C_{1}^{c}\cap{C_{2}}$.\par
Now for use in our algorithm we label these subsets. We will refer to the total number of subsets $P$ as $\Theta$ and label these subsets as $\Ptt$ where $\theta=1,\dots,\Theta$. We can write the elements of each of these sets $\Ptt$ as $\pitt$, with $i=1,\dots,\cPtt$ where $\cPtt:=|\Ptt|$. In Figure \ref{fig:intersectExample} we have $\Theta=3$, $P^{1}=\{1\}$ with $p_{1}^{1}=1$, $P^{2}=\{1,2\}$ with $p_{1}^{2}=1$ and $p_{2}^{2}=2$, and $P^{3}=\{2\}$ with $p_{1}^{3}=2$.
\begin{figure}
    \centering
    
     
    \tikzset{
    pattern size/.store in=\mcSize, 
    pattern size = 5pt,
    pattern thickness/.store in=\mcThickness, 
    pattern thickness = 0.3pt,
    pattern radius/.store in=\mcRadius, 
    pattern radius = 1pt}
    \makeatletter
    \pgfutil@ifundefined{pgf@pattern@name@_fsapumgd2}{
    \pgfdeclarepatternformonly[\mcThickness,\mcSize]{_fsapumgd2}
    {\pgfqpoint{0pt}{0pt}}
    {\pgfpoint{\mcSize+\mcThickness}{\mcSize+\mcThickness}}
    {\pgfpoint{\mcSize}{\mcSize}}
    {
    \pgfsetcolor{\tikz@pattern@color}
    \pgfsetlinewidth{\mcThickness}
    \pgfpathmoveto{\pgfqpoint{0pt}{0pt}}
    \pgfpathlineto{\pgfpoint{\mcSize+\mcThickness}{\mcSize+\mcThickness}}
    \pgfusepath{stroke}
    }}
    \makeatother
    
     
    \tikzset{
    pattern size/.store in=\mcSize, 
    pattern size = 5pt,
    pattern thickness/.store in=\mcThickness, 
    pattern thickness = 0.3pt,
    pattern radius/.store in=\mcRadius, 
    pattern radius = 1pt}
    \makeatletter
    \pgfutil@ifundefined{pgf@pattern@name@_hqemfbp3e}{
    \pgfdeclarepatternformonly[\mcThickness,\mcSize]{_hqemfbp3e}
    {\pgfqpoint{0pt}{-\mcThickness}}
    {\pgfpoint{\mcSize}{\mcSize}}
    {\pgfpoint{\mcSize}{\mcSize}}
    {
    \pgfsetcolor{\tikz@pattern@color}
    \pgfsetlinewidth{\mcThickness}
    \pgfpathmoveto{\pgfqpoint{0pt}{\mcSize}}
    \pgfpathlineto{\pgfpoint{\mcSize+\mcThickness}{-\mcThickness}}
    \pgfusepath{stroke}
    }}
    \makeatother
    
     
    \tikzset{
    pattern size/.store in=\mcSize, 
    pattern size = 5pt,
    pattern thickness/.store in=\mcThickness, 
    pattern thickness = 0.3pt,
    pattern radius/.store in=\mcRadius, 
    pattern radius = 1pt}
    \makeatletter
    \pgfutil@ifundefined{pgf@pattern@name@_hay99zx4i}{
    \pgfdeclarepatternformonly[\mcThickness,\mcSize]{_hay99zx4i}
    {\pgfqpoint{0pt}{0pt}}
    {\pgfpoint{\mcSize+\mcThickness}{\mcSize+\mcThickness}}
    {\pgfpoint{\mcSize}{\mcSize}}
    {
    \pgfsetcolor{\tikz@pattern@color}
    \pgfsetlinewidth{\mcThickness}
    \pgfpathmoveto{\pgfqpoint{0pt}{0pt}}
    \pgfpathlineto{\pgfpoint{\mcSize+\mcThickness}{\mcSize+\mcThickness}}
    \pgfusepath{stroke}
    }}
    \makeatother
    
     
    \tikzset{
    pattern size/.store in=\mcSize, 
    pattern size = 5pt,
    pattern thickness/.store in=\mcThickness, 
    pattern thickness = 0.3pt,
    pattern radius/.store in=\mcRadius, 
    pattern radius = 1pt}
    \makeatletter
    \pgfutil@ifundefined{pgf@pattern@name@_his1blmqs}{
    \pgfdeclarepatternformonly[\mcThickness,\mcSize]{_his1blmqs}
    {\pgfqpoint{0pt}{-\mcThickness}}
    {\pgfpoint{\mcSize}{\mcSize}}
    {\pgfpoint{\mcSize}{\mcSize}}
    {
    \pgfsetcolor{\tikz@pattern@color}
    \pgfsetlinewidth{\mcThickness}
    \pgfpathmoveto{\pgfqpoint{0pt}{\mcSize}}
    \pgfpathlineto{\pgfpoint{\mcSize+\mcThickness}{-\mcThickness}}
    \pgfusepath{stroke}
    }}
    \makeatother
    
     
    \tikzset{
    pattern size/.store in=\mcSize, 
    pattern size = 5pt,
    pattern thickness/.store in=\mcThickness, 
    pattern thickness = 0.3pt,
    pattern radius/.store in=\mcRadius, 
    pattern radius = 1pt}
    \makeatletter
    \pgfutil@ifundefined{pgf@pattern@name@_l2pcqf36c}{
    \pgfdeclarepatternformonly[\mcThickness,\mcSize]{_l2pcqf36c}
    {\pgfqpoint{0pt}{0pt}}
    {\pgfpoint{\mcSize+\mcThickness}{\mcSize+\mcThickness}}
    {\pgfpoint{\mcSize}{\mcSize}}
    {
    \pgfsetcolor{\tikz@pattern@color}
    \pgfsetlinewidth{\mcThickness}
    \pgfpathmoveto{\pgfqpoint{0pt}{0pt}}
    \pgfpathlineto{\pgfpoint{\mcSize+\mcThickness}{\mcSize+\mcThickness}}
    \pgfusepath{stroke}
    }}
    \makeatother
    
     
    \tikzset{
    pattern size/.store in=\mcSize, 
    pattern size = 5pt,
    pattern thickness/.store in=\mcThickness, 
    pattern thickness = 0.3pt,
    pattern radius/.store in=\mcRadius, 
    pattern radius = 1pt}
    \makeatletter
    \pgfutil@ifundefined{pgf@pattern@name@_7jwpffzob}{
    \pgfdeclarepatternformonly[\mcThickness,\mcSize]{_7jwpffzob}
    {\pgfqpoint{0pt}{-\mcThickness}}
    {\pgfpoint{\mcSize}{\mcSize}}
    {\pgfpoint{\mcSize}{\mcSize}}
    {
    \pgfsetcolor{\tikz@pattern@color}
    \pgfsetlinewidth{\mcThickness}
    \pgfpathmoveto{\pgfqpoint{0pt}{\mcSize}}
    \pgfpathlineto{\pgfpoint{\mcSize+\mcThickness}{-\mcThickness}}
    \pgfusepath{stroke}
    }}
    \makeatother
    \tikzset{every picture/.style={line width=0.75pt}} 
    
    \begin{tikzpicture}[x=0.75pt,y=0.75pt,yscale=-1,xscale=1]
    
    \draw  [line width=1.5]  (33.05,33.1) -- (165.55,33.1) -- (165.55,121.44) -- (33.05,121.44) -- cycle ;
    \draw  [pattern=_fsapumgd2,pattern size=9pt,pattern thickness=0.75pt,pattern radius=0pt, pattern color={rgb, 255:red, 0; green, 0; blue, 0}] (33.05,33.1) .. controls (33.32,33.23) and (141.63,32.81) .. (142.36,33.54) .. controls (143.09,34.27) and (81.81,46.67) .. (84.73,62.72) .. controls (87.64,78.77) and (148.19,83.15) .. (144.54,97.01) .. controls (140.9,110.87) and (125.49,121.5) .. (124.92,121.5) .. controls (124.36,121.5) and (32.38,120.94) .. (33.05,121.44) .. controls (33.72,121.94) and (32.78,32.97) .. (33.05,33.1) -- cycle ;
    \draw  [pattern=_hqemfbp3e,pattern size=9pt,pattern thickness=0.75pt,pattern radius=0pt, pattern color={rgb, 255:red, 0; green, 0; blue, 0}] (33.05,33.1) .. controls (33.66,33.54) and (99.57,33.37) .. (100.1,33.37) .. controls (100.64,33.37) and (165.95,33.11) .. (165.55,33.1) .. controls (165.15,33.1) and (165.95,121.71) .. (165.55,121.44) .. controls (165.15,121.17) and (110.99,121.08) .. (110.26,121.08) .. controls (109.53,121.08) and (66.49,99.2) .. (71.59,78.77) .. controls (76.7,58.34) and (32.93,53.97) .. (32.93,53.97) .. controls (32.93,53.97) and (32.44,32.66) .. (33.05,33.1) -- cycle ;
    \draw  [pattern=_hay99zx4i,pattern size=9pt,pattern thickness=0.75pt,pattern radius=0pt, pattern color={rgb, 255:red, 0; green, 0; blue, 0}][line width=0.75]  (174.45,39.01) -- (189.04,39.01) -- (189.04,53.6) -- (174.45,53.6) -- cycle ;
    \draw  [pattern=_his1blmqs,pattern size=9pt,pattern thickness=0.75pt,pattern radius=0pt, pattern color={rgb, 255:red, 0; green, 0; blue, 0}][line width=0.75]  (174.45,97.37) -- (189.04,97.37) -- (189.04,111.96) -- (174.45,111.96) -- cycle ;
    \draw   (174.45,68.19) -- (189.04,68.19) -- (189.04,82.78) -- (174.45,82.78) -- cycle ;
    \draw  [pattern=_l2pcqf36c,pattern size=9pt,pattern thickness=0.75pt,pattern radius=0pt, pattern color={rgb, 255:red, 0; green, 0; blue, 0}] (174.45,68.19) -- (189.04,68.19) -- (189.04,82.78) -- (174.45,82.78) -- cycle ;
    \draw  [pattern=_7jwpffzob,pattern size=9pt,pattern thickness=0.75pt,pattern radius=0pt, pattern color={rgb, 255:red, 0; green, 0; blue, 0}][line width=0.75]  (174.45,68.19) -- (189.04,68.19) -- (189.04,82.78) -- (174.45,82.78) -- cycle ;
    
    \draw (190.41,65.93) node [anchor=north west][inner sep=0.75pt]  [font=\footnotesize]  {$=C_{1} \cap C_{2} =\Omega ^{\{1,2\}} =\Omega {^{P}}^{2}$};
    \draw (151.8,18.41) node [anchor=north west][inner sep=0.75pt]    {$\Omega $};
    \draw (191.36,36.56) node [anchor=north west][inner sep=0.75pt]  [font=\footnotesize]  {$=C_{1} \cap C_{2}^{c} =\Omega ^{\{1\}} =\Omega {^{P}}^{1}$};
    \draw (191.36,94.92) node [anchor=north west][inner sep=0.75pt]  [font=\footnotesize]  {$=C_{1}^{c} \cap C_{2} =\Omega ^{\{2\}} =\Omega {^{P}}^{3}$};

    \end{tikzpicture}

    \caption{An example state space $\Omega$ with two functions $\mathbf{f}_{1}$ and $\mathbf{f}_{2}$ with domains $C_{1}$ and $C_{2}$, respectively. This divides the whole state space into three separate subsets using definition (\ref{eqnOPDef}), the intersection of the two domains and the two sets where either only $\mathbf{f}_{1}$ or only $\mathbf{f}_{2}$ is defined. This tells us that $\Theta=3$ and we label $P^{1}=\{1\}$, $P^{2}=\{1,2\}$ and $P^{3}=\{2\}$.} 
    \label{fig:intersectExample}
\end{figure}
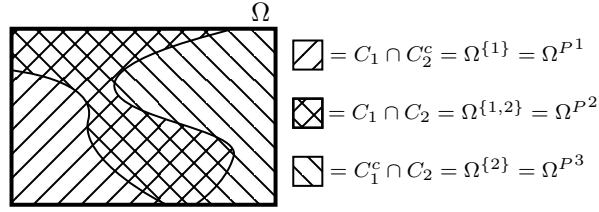

\subsection{Reproducing Kernel Hilbert Spaces}
\label{secRKHS}

We construct a function $v$ as a solution of the minimisation problem (\ref{eqnMinProbv}). It is the optimal reconstruction of a function $V$ satisfying (\ref{eqnDV}), which is a Lyapunov function for our system (\ref{eqnfis}). In the same way as \cite{giesl2018construction}, $v$ will belong to a Reproducing Kernel Hilbert space $H$ 
with reproducing kernel $\Phi:\Omega\times\Omega\rightarrow\bR$.
\begin{defn}{\rm (Reproducing Kernel Hilbert Space (RKHS))}
    \label{defnRKHS}
    Let $\Omega\subseteq\bR^{d}$. A Hilbert space $H=H(\Omega)$ of continuous functions $f:\Omega\rightarrow\bR$ with inner product $\langle\cdot,\cdot\rangle_{H}$ is called a reproducing kernel Hilbert space, abbreviated RKHS, if there is a function $\Phi:\Omega\times\Omega\rightarrow\bR$ with
    \begin{enumerate}
        \item $\Phi(\cdot,\mathbf{x})\in{H}$ for all ${\mathbf{x}\in\Omega}$
        \item $g(\mathbf{x})=\langle{g(\cdot),\Phi(\cdot,\mathbf{x})}\rangle_{H}$ for all $g\in{H}$ and all $\mathbf{x}\in\Omega$.
    \end{enumerate}
    The function $\Phi$ is called the reproducing kernel of $H$. The reproducing kernel is called positive definite, if for any pair of pairwise distinct points $\{x_{1},...,x_{N}\}\subseteq\Omega$ the matrix $(\Phi(x_{i},x_{j}))_{i,j=1,...,N}$ is positive definite.
\end{defn}
In our case, the reproducing kernel is given by a radial basis function $\Psi_{0}$ through $\Phi(\mathbf{x},\mathbf{y}):=\Psi_{0}(\|\mathbf{x}-\mathbf{y}\|_{2})$. We choose this radial basis function to be a Wendland function $\phi_{\ell,k}(cr)$ (\cite{wendland1995piecewise}; \cite{wendland1998error}) where $c>0$, $k\in\bN$ is a smoothness parameter and $\ell=\lfloor\frac{d}{2}\rfloor+k+1$. These are smooth functions with compact support, which are polynomials on their support. We can then define recursively $\Psi_{i}(r)=\frac{1}{r}\frac{d\Psi_{i-1}(r)}{dr}$ for $i=1,2$ and $r>0$. Note that $\lim_{r\rightarrow{0}}\Psi_{i}(r)$ exists if the smoothness parameter $k$ of the Wendland function is sufficiently large.

\subsection{Minimisation problem}
\label{secMinimise}

Using this set-up, we will construct a Lyapunov function for our switched system using quadratic programming, as done for autonomous ODEs in \cite{giesl2018construction}. There they construct a Lyapunov function for a single system by finding a solution to the system of partial differential inequalities $DV(\mathbf{x})\leq{r}(\mathbf{x})$ where $D$ denotes the orbital derivative, and $r(\mathbf{x})$ is negative. This method can be generalised in order to construct a Lyapunov function for a switched system. We define for all $\theta=1\dots,\Theta$ and $i=1,\dots,\cPtt$ the linear differential operator
\begin{equation}
    \label{eqnDpitt}
    D_{\pitt}V(\mathbf{x}):=\nabla{
    V}(\mathbf{x})\cdot\mathbf{f}_{\pitt}(\mathbf{x})\text{ }\forall\mathbf{x}\in\OPtt.
\end{equation}
Our system of partial differential inequalities is then written as
\begin{equation}
    \label{eqnDV}
    D_{\pitt}V(\mathbf{x})\leq{b_{\pitt}}(\mathbf{x})
\end{equation}
where $b_{\pitt}(\mathbf{x})$ is negative for all $\mathbf{x}\in\OPtt\setminus{\{\mathbf{0}\}}$, $b_{\pitt}(\mathbf{0})=0$. Finding a function that satisfies these inequalities is then equivalent to finding a Lyapunov function as in (\ref{eqnLyapSwitch}).\par
We now discretise our problem by choosing collocation points
\begin{equation}
    \label{eqnCollPts}
    \mathbf{x}^{1,\theta},\dots,\mathbf{x}^{\Ntt,\theta}\in\OPtt\nonumber
\end{equation}
and values
\begin{equation}
\label{eqnValb}
    b_{i}^{1,\theta}=b_{\pitt}(\mathbf{x}^{1,\theta}),\dots,b_{i}^{\Ntt,\theta}=b_{\pitt}(\mathbf{x}^{\Ntt,\theta})\in\bR_{0}^{+}.\nonumber
\end{equation}
We seek to find a function $v$ that satisfies
\begin{equation}
    \label{eqnDiscretProb}
    D_{\pitt}v(\xltt)\leq\biltt\text{ for all }\ell=1,\dots,\Ntt
\end{equation}
for $\theta=1,\dots,\Theta$ and $i=1,\dots,\cPtt$, where $\Ntt$ is the number of collocation points chosen in the subset $\OPtt$.\par
Another way to write this that is more suitable for our purpose would be to say we look to find an appropriate function $v\in{H}$ using the information $\biltt\in\bR$ given by the functionals $\lliltt\in{H}^{*}$, i.e. $\lliltt(v)\leq\biltt$ for all $\theta=1,\dots,\Theta$, $i=1,\dots,\cPtt$ and $\ell=1,\dots,\Ntt$. These functionals are defined as our differential operator (\ref{eqnDpitt}) applied at a specific collocation point, i.e.
\begin{equation}
    \label{eqnlliltt}
    (\lliltt)v(\cdot):=(\delta_{\xltt}\circ{D_{\pitt}})v(\cdot)=\nabla{v}(\xltt)\cdot\mathbf{f}_{\pitt}(\xltt)
\end{equation}
for all $\theta=1,\dots,\Theta$, $\ell=1,\dots,\Ntt$ and $i=1,\dots,\cPtt$, with $\delta$ being Dirac's delta distribution. Note that the elements of the set $\Ptt$ will be reordered for each collocation point in the following section when we select a basis of the functions $\mathbf{f}_{i}$ at this point, see (\ref{eqnPltt}).\par
The optimal reconstruction of $v$ is given by the solution of the problem
\begin{equation}
    \label{eqnMinProbv}
    \min\{\|v\|_{H}:v\in{H}\text{, }\lliltt(v)\leq\biltt\text{ for all }\theta=1,\dots,\Theta\text{, }\ell=1,\dots,\Ntt\text{ and }i=1,\dots,\cPtt\}
\end{equation}
It can be shown that this problem has at most one solution, using a similar proof to that of Lemma 5.1 in \cite{giesl2018construction}.

\subsection{Selection of Basis}
\label{secBasis}

If the functionals are linearly independent, we can use techniques from \cite{wendland2004scattered} to find a solution of the problem (\ref{eqnMinProbv}). However, this is not necessarily the case. We fix a point $\xltt\in\OPtt$, consider whether $\mathbf{f}_{\pitt}(\xltt)$ for $i=1,\dots,\cPtt$ are linearly independent, and if not, select a basis of these vectors.\par 
To this end, fix $\xltt\in\OPtt$ and define the set 
\begin{equation}
    \label{eqnRltt}
    \Rltt:=\{r_{1}^{\ell,\theta},\dots,r_{\rrltt}^{\ell,\theta}\}\subseteq\Ptt\nonumber
\end{equation}
to be a set of indices such that the vectors $\mathbf{f}_{\riltt}(\xltt)$ are a basis of $(\mathbf{f}_{\pitt}(\xltt))_{\pitt\in\Ptt}$ (with $\rrltt$ being the rank of this matrix). The elements of $\Ptt\setminus{\Rltt}$ will be labeled $\pi_{s}^{\ell,\theta}$ where $s=1,\dots,\cPtt-\rrltt$.\par
This means we can consider the vector 
\begin{equation}
    \label{eqnPltt}
    \Pltt=\big(r_{1}^{\ell,\theta},\dots,r_{\rrltt}^{\ell,\theta},\pi_{1}^{\ell,\theta},\dots,\pi_{\cPtt-\rrltt}^{\ell,\theta}\big)
\end{equation}
and label these coordinates $\piltt\in\Pltt$ for $i=1,\dots,\cPtt$. This means we can represent an ordering of the set $\Ptt$ for the point $\xltt$, where we have placed the indices that represent the basis at the beginning of the vector. The vector (\ref{eqnPltt}) is found for each of our collocation points $\xltt$ for $\theta=1,\dots,\Theta$ and $\ell=1,\dots,\Ntt$.\par
From this point forwards when we refer to the functionals $\lliltt$ we mean (\ref{eqnlliltt}) with $\pitt$ instead being elements $\piltt$ from the vector $\Pltt$ (\ref{eqnPltt}).\par
For notation purposes we define
\begin{equation}
    \label{eqnRhocP}
    \rho:=\sum_{\theta=1}^{\Theta}\sum_{\ell=1}^{\Ntt}\rrltt\text{ and }\cP:=\sum_{\theta=1}^{\Theta}\sum_{\ell=1}^{\Ntt}\cPtt.\nonumber
\end{equation}
By (\ref{eqnPltt}) the functionals $\lliltt$ for $i=1,\dots,\rrltt$ are a basis for the whole set of functionals with $i=1,\dots,\cPtt$. The functionals at the point $\xltt$ can be expressed as a linear combination of this basis
\begin{equation}
    \label{eqnllssltt}
    \llssltt=\sum_{i=1}^{\rrltt}\aassiltt\lliltt
\end{equation}
for all $\sigma=1,\dots,\cPtt$, for some appropriate $\aassiltt\in\bR$. We assemble the coefficients of all collocation points $\xltt$ 
as elements of a matrix $Q\in\bR^{\cP\times\rho}$ in the following way
\begin{equation}
    \label{eqnQ}
    Q:=\left(\begin{array}{c}
    I\\
    \Tilde{Q}
    \end{array}\right)
\end{equation}
where $I$ is the $\rho\times\rho$ identity matrix (accounting for the elements $\alpha_{s,i}^{\ell,\theta}$ for $\theta=1,\dots,\Theta$, $\ell=1,\dots,\Ntt$, $s,i=1,\dots,\rrltt$) and $\Tilde{Q}$ is a block diagonal matrix 
\begin{equation}
    \label{eqnQdiag}
    \Tilde{Q}:=\diag(\Tilde{Q}^{1,1},\dots,\Tilde{Q}^{N^{1},1},\dots,\Tilde{Q}^{1,\Theta},\dots,\Tilde{Q}^{N^{\Theta},\Theta})
\end{equation}
where the matrices $\Tilde{Q}^{\ell,\theta}$ have entries $q^{\ell,\theta}_{s,i}:=\alpha_{s+\rrltt,i}^{\ell,\theta}$ for $\theta=1,\dots,\Theta$, $\ell=1,\dots,\Ntt$, $s=1,\dots,\cPtt-\rrltt$ and $i=1,\dots,\rrltt$.

\subsection{Quadratic Programming Problem}
\label{secQuadProg}

We show that the solution of (\ref{eqnMinProbv}) is of the form
\begin{equation}
    \label{eqnLyapv}
    v(\mathbf{x})=\sum_{\theta=1}^{\Theta}\sum_{\ell=1}^{\Ntt}\sum_{i=1}^{\rrltt}\bbiltt(\lliltt)^{\mathbf{y}}\Phi(\mathbf{x},\mathbf{y})
\end{equation}
with $(\lliltt)^{\mathbf{y}}\Phi(\mathbf{x},\mathbf{y})=\nabla_{\mathbf{y}}({\Phi}(\mathbf{x},\mathbf{y}))\big|_{\mathbf{y}=\xltt}\cdot\mathbf{f}_{\piltt}(\xltt)$, where the coefficients $\bbiltt$ satisfy
\begin{equation}
    \label{eqnbTAb}
    \begin{cases}
        \text{\rm{minimise }}\beta^{T}A\beta\\
        \text{\rm{subject to }}QA\beta\leq\mathbf{b}
    \end{cases}
\end{equation}
To find these matrices, we have applied the functionals to our proposed solution (\ref{eqnLyapv}) and used the inequalities in (\ref{eqnMinProbv}).
The elements of the matrix $A\in\bR^{\rho\times\rho}$ are
\begin{eqnarray}
    \aijtteelm & := & (\delta_{\xeett}\circ{D_{\pjmee}})^{\mathbf{x}}(\delta_{\xltt}\circ{D_{\piltt}})^{\mathbf{y}}\Phi(\mathbf{x},\mathbf{y})\nonumber\\
    \label{eqnAelem}
    & = & \langle(\lljmee)^{\mathbf{y}}\Phi(\cdot,\mathbf{y}),(\lliltt)^{\mathbf{x}}\Phi(\cdot,\mathbf{x})\rangle_{H}
\end{eqnarray}
for $\theta,\eta=1,\dots,\Theta$, $\ell=1,\dots,\Ntt$, $m=1,\dots,\Nee$, $i=1,\dots,\rrltt$ and $j=1,\dots,\rrmee$. The symmetric matrix $A$ accounts for the linearly independent functionals, and we apply the matrix $Q$ as the other functionals are expressed as linear combinations of the independent ones (\ref{eqnllssltt}). The elements of the vector $\mathbf{b}$ are the values $\biltt$ from the minimisation problem (\ref{eqnMinProbv}) for $\theta=1,\dots,\Theta$, $\ell=1,\dots,\Ntt$ and $i=1,\dots,\cPtt$.\par
Similarly to \cite{giesl2018construction}, minimising $\beta^{T}A\beta$ is the same as minimising $\|v\|_{H}$.\par
The following lemma contributes to Theorem \ref{lemMain}.
\begin{lemma}
    \label{lemUnibTAb}
    If there is a feasible solution to the problem (\ref{eqnbTAb}), then this solution is unique.
\end{lemma}
\begin{proof}
    Given the definition of $Q$ as in (\ref{eqnQ}), we can rewrite problem (\ref{eqnbTAb}) as follows
    \begin{equation}
        \label{lemUnibTAbproofeqn1}
        \begin{cases}
            \text{\rm{minimise }}\beta^{T}A\beta\\
            \text{\rm{subject to }}A\beta=:\mathbf{r}\leq\mathbf{c}\\
            \phantom{xxxxxxxx}\Tilde{Q}A\beta=:\Tilde{\mathbf{r}}\leq\Tilde{\mathbf{c}}
        \end{cases}
    \end{equation}
    where, for $\theta=1,\dots,\Theta$ and $\ell=1,\dots,\Ntt$, elements of $\mathbf{c}$ are $\biltt$ for $i=1,\dots,\rrltt$, and elements of $\Tilde{\mathbf{c}}$ are $b_{s+\rrltt}^{\ell,\theta}$ for $s=1,\dots,\cPtt-\rrltt$.\par
    The matrix $A$ is positive definite since the functionals included in the values of its elements (\ref{eqnAelem}) are linearly independent. As $A$ is also symmetric, we  have 
     $\beta^{T}A\beta=\mathbf{r}^{T}A^{-1}\mathbf{r}$ and, using that $\Tilde{Q}\mathbf{r}=\Tilde{\mathbf{r}}$, we can rewrite (\ref{lemUnibTAbproofeqn1}) as
    \begin{equation}
        \label{lemUnibTAbproofeqn2}
        \begin{cases}
            \text{\rm{minimise }}\mathbf{r}^{T}A^{-1}\mathbf{r}\\
            \text{\rm{subject to }}\mathbf{r}-\mathbf{c}\leq\mathbf{0}\\
            \phantom{xxxxxxxx}\Tilde{Q}\mathbf{r}-\Tilde{\mathbf{c}}\leq{\mathbf{0}}
        \end{cases}
    \end{equation}
    The functions $g_{1}(\mathbf{r}):=\mathbf{r}-\mathbf{c}$ and $g_{2}(\mathbf{r}):=\Tilde{Q}\mathbf{r}-\Tilde{\mathbf{c}}$ which correspond to the constraints of (\ref{lemUnibTAbproofeqn2}) are both affine and therefore convex. As $A^{-1}$ is positive definite, the function $h(\mathbf{r})=\mathbf{r}^{T}A^{-1}\mathbf{r}$ is strictly convex. Overall this means, using the definition of a convex optimisation problem from \cite{Boyd_Vandenberghe_2004}, that (\ref{lemUnibTAbproofeqn2}) has a unique solution if it is feasible.
\end{proof}
The main result of this section is then as follows.
\begin{theorem}
    \label{lemMain}
    If there is a feasible solution to the problem (\ref{eqnbTAb}), then the unique minimiser of the problem (\ref{eqnMinProbv}) is of the form (\ref{eqnLyapv}) where the coefficients $\bbiltt$ are determined by the solution of the problem (\ref{eqnbTAb}).
\end{theorem}
\begin{proof}
    Define $v$ by (\ref{eqnLyapv}) where the coefficients are determined by (\ref{eqnbTAb}) (see Lemma \ref{lemUnibTAb}). Let $s\in{H}$ be any function satisfying the constraints of the minimisation problem (\ref{eqnMinProbv}). We wish to show that $\|v\|_{H}\leq\|s\|_{H}$. We know that $s$ satisfies
    \begin{equation}
        \label{lemMainproofeqn1}
        \lliltt(s)=:r_{i}^{\ell,\theta}\leq\biltt
    \end{equation}
    for all $\theta=1,\dots,\Theta$, $\ell=1,\dots,\Ntt$ and $i=1,\dots,\cPtt$. For the rest of the lemma we set elements of $\mathbf{r}$ to be $r_{i}^{\ell,\theta}$ for $i=1,\dots,\rrltt$, and elements of $\Tilde{\mathbf{r}}$ to be $r_{s+\rrltt}^{\ell,\theta}$ for $s=1,\dots,\cPtt-\rrltt$.\par
    Consider the new minimisation problem
    \begin{equation}
        \label{lemMainproofeqn2}
        \begin{cases}
            \text{\rm{minimise }}\|v\|_{H}\\
            \text{\rm{subject to }}\lliltt(v)=r_{i}^{\ell,\theta}\text{, }\forall\theta=1,\dots,\Theta\text{, }\ell=1,\dots,\Ntt\\
            \phantom{ccccccccccccccccccccccccc}\text{ and }i=1,\dots,\rrltt
        \end{cases}\nonumber
    \end{equation}
    Now using Theorem 16.1 from \cite{wendland2004scattered} we can see that the solution of this problem is given by
    \begin{equation}
        \label{lemMainproofeqn3}
        \Tilde{s}(\mathbf{x})=\sum_{\theta=1}^{\Theta}\sum_{\ell=1}^{\Ntt}\sum_{i=1}^{\rrltt}\Tilde{\beta}_{i}^{\ell,\theta}(\lliltt)^{\mathbf{y}}\Phi(\cdot,\mathbf{y})\nonumber
    \end{equation}
    where $A\Tilde{\beta}=\mathbf{r}$, $A$ has entries (\ref{eqnAelem}), and $\Tilde{\beta}$ has the same structure as $\beta$ from (\ref{eqnbTAb}). We will show that these $\Tilde{\beta}$ satisfy $\Tilde{Q}A\Tilde{\beta}=\Tilde{\mathbf{r}}$, which is equivalent to showing that $\Tilde{Q}\mathbf{r}=\Tilde{\mathbf{r}}$. Indeed, using (\ref{eqnllssltt}) we can rewrite the equations in (\ref{lemMainproofeqn1}) for $\sigma=\rrltt+1,\dots,\cPtt$ as 
    \begin{equation}
        \label{lemMainproofeqn4}
        \sum_{i=1}^{\rrltt}\alpha_{\sigma,i}^{\ell,\theta}\lambda_{i}^{\ell,\theta}(s)=r_{\sigma}^{\ell,\theta}\Rightarrow\sum_{i=1}^{\rrltt}\alpha_{\sigma,i}^{\ell,\theta}r_{i}^{\ell,\theta}=r_{\sigma}^{\ell,\theta}\Rightarrow\Tilde{Q}\mathbf{r}=\Tilde{\mathbf{r}}\nonumber
    \end{equation}
    using the definition of the elements of $\Tilde{Q}$. Therefore $\Tilde{s}$ is also the minimiser for the problem
    \begin{equation}
        \label{lemMainproofeqn5}
        \begin{cases}
            \text{\rm{minimise }}\|v\|_{H}\\
            \text{\rm{subject to }}\lliltt(v)=r_{i}^{\ell,\theta}\text{, }\forall\theta=1,\dots,\Theta\text{, }\ell=1,\dots,\Ntt\\
            \phantom{ccccccccccccccccccccccccc}\text{ and }i=1,\dots,\cPtt
        \end{cases}\nonumber
    \end{equation}
    which shows that $\|\Tilde{s}\|_{H}\leq\|s\|_{H}$.\par
    Now both $v$ and $\Tilde{s}$ are of the form (\ref{eqnLyapv}) and the coefficients $\beta$ and $\Tilde{\beta}$ both satisfy the constraints of problem (\ref{eqnbTAb}). However, the coefficients $\beta$ of $v$ minimise $\beta^{T}A\beta=\|v\|_{H}^{2}$ so that $\|v\|_{H}^{2}\leq\|\Tilde{s}\|_{H}^{2}=\Tilde{\beta}^{T}A\Tilde{\beta}$. Therefore altogether we have that
    \begin{equation}
        \label{lemMainproofeqn6}
        \|v\|_{H}\leq\|\Tilde{s}\|_{H}\leq\|s\|_{H}.\nonumber
    \end{equation}
    Since the minimiser is unique (see statement after \eqref{eqnMinProbv}),
     $v$ is the minimiser.
\end{proof}

This means that the function $v$ from (\ref{eqnLyapv}) is an approximation of a function that solves the system of partial differential inequalities in (\ref{eqnDV}). If the matrix $Q$ is just the $\rho\times\rho$ identity matrix, then we are able to verify that the problem (\ref{eqnbTAb}) is feasible, this is shown in \cite{giesl2018construction}.

\section{Examples}
\label{secEx}

This section contains three examples that show how our algorithm can be applied to systems that use different switching rules. For arbitrary switched see Example \ref{exArb}, for variable structure see Example \ref{exVarStruct}, and for a mix of the two see Example \ref{exCombo}.\par
In each of these examples the Wendland functions that we use in the implementation of Algorithm \ref{algMain} is $\phi_{4,2}$ from \cite{giesl2007constructionRB} which means that we have
\begin{equation}
    \Psi_{0}(r)=(1-cr)^{6}_{+}[35(cr)^{2}+18cr+3]
\end{equation}
from which $\Psi_{1}$ and $\Psi_{2}$ can be calculated. We use $c=\frac{5}{6}$. The collocation points are selected to be $\frac{1}{6}\bZ^{2}\cap\Omega$, where $\Omega$ is $[-0.5,0.5]^{2}$. This means we have 49 collocation points in total. The computational complexity is of the order $N^2$ where $N$ is the number of collocation points.

\begin{figure}[ht]
    \centering
    \includegraphics[width=0.32\linewidth]{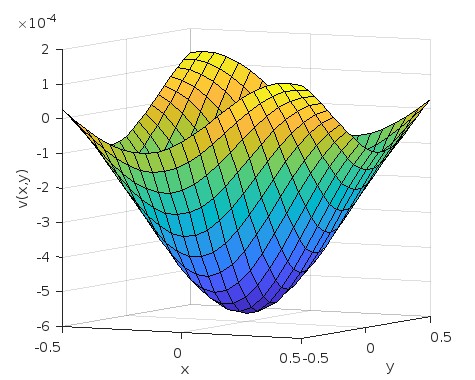}
    \includegraphics[width=0.32\linewidth]{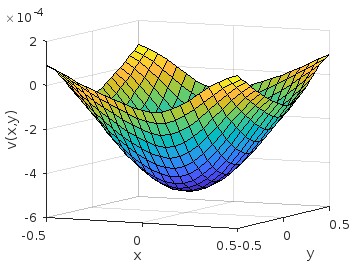}
    \includegraphics[width=0.32\linewidth]{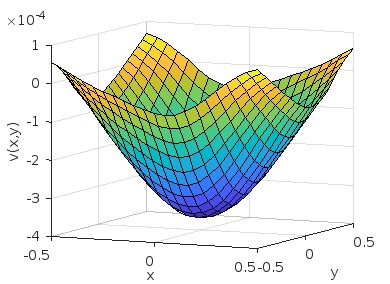}
    \caption{Lyapunov functions constructed for the following three examples. The left function is for Example \ref{exArb}, the middle is for Example \ref{exVarStruct} and the right is for Example \ref{exCombo}. The left and right functions took $3$ minutes to compute, and the middle $45$ seconds.}
    \label{fig:Lyaps}
\end{figure}

\begin{example}{\rm (Arbitrary Switched)}
\label{exArb}

Using $\mathbf{f}_{1}$ and $\mathbf{f}_{2}$ from Example 9.2 in \cite{hafstein2009algorithm}, we consider the autonomous systems
\begin{equation}
\label{eqnArbExf1}
\mathbf{\dot{x}}=\mathbf{f}_{1}(\mathbf{x})\text{, where }\mathbf{f}_{1}(x,y):=\left(\begin{array}{c}
-y \\
x-y(1-x^{2}+0.1x^{4})
\end{array}\right)\nonumber
\end{equation}
and
\begin{equation}
\label{eqnArbExf2}
\mathbf{\dot{x}}=\mathbf{f}_{2}(\mathbf{x})\text{, where }\mathbf{f}_{2}(x,y):=\left(\begin{array}{c}
-y+x(x^{2}+y^{2}-1)\\
x+y(x^{2}+y^{2}-1)
\end{array}\right)\text{.}\nonumber
\end{equation}
These are both asymptotically stable systems, and we switch between them arbitrarily over the whole state space $\Omega$. This can be expressed using the following system
\begin{equation}
    \label{eqnArbSwitched}
    \dot{\mathbf{x}}=\mathbf{f}_{p}(\mathbf{x})\text{, }p\in\{1,2\}.
\end{equation}
We set
 $C_{1}=C_{2}=\Omega$, and find that $\Theta=1$, $P^{1}=\{1,2\}$ with $\Omega^{P^{1}}=\Omega$. The Lyapunov function produced for this system can be seen on the left of Figure \ref{fig:Lyaps} and its orbital derivatives in Figure \ref{fig:ArbOrbDervs}, which are negative apart from the origin.

\begin{figure}[ht]
    \centering
    \includegraphics[width=0.44\linewidth]{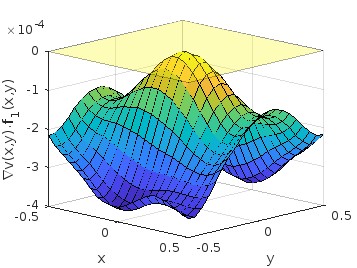}
    \includegraphics[width=0.44\linewidth]{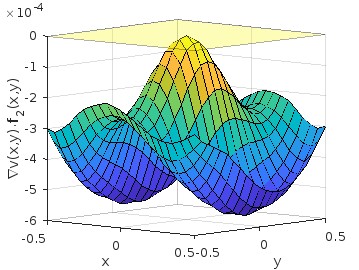}
    \caption{Plots of the orbital derivatives of the Lyapunov function displayed on the left of Figure \ref{fig:Lyaps}: $\nabla{v}(x,y)\cdot\mathbf{f}_{1}(x,y)$ (left) and $\nabla{v}(x,y)\cdot\mathbf{f}_{2}(x,y)$ (right). In both figures the plane $\{z=0\}$ has been plotted to illustrate that the constraints in (\ref{eqnMinProbv}) have been met.}
    \label{fig:ArbOrbDervs}
\end{figure}
\end{example}

\begin{example}{\rm (Variable Structure)}
\label{exVarStruct}
Here we consider the example found in \cite{liberzon1999basic} (also studied in \cite{hafstein2009algorithm}), with the two unstable linear systems $\mathbf{\dot{x}}=A_{1}\mathbf{x}$ and $\mathbf{\dot{x}}=A_{2}\mathbf{x}$, where
\begin{equation}
   \label{eqnVarStrucA1}
   A_{1}:=\left(\begin{array}{cc}
   0.1 & -1 \\
   2 & 0.1
   \end{array}\right)\text{ and }A_{2}:=\left(\begin{array}{cc}
   0.1 & -2 \\
   1 & 0.1
   \end{array}\right).
\end{equation}
We can switch between these systems in such a way that we can produce a variable structure system where trajectories are asymptotically attracted to the equilibrium. This occurs when we define the following quadrants
\begin{eqnarray}
Q_{1} & := & \{(x,y):x>0\text{ and }y\geq{0}\}\cap\{(0,0)\}\nonumber\\
Q_{2} & := & \{(x,y):x\leq{0}\text{ and }y>0\}\cap\{(0,0)\}\nonumber\\
Q_{3} & := & \{(x,y):x<0\text{ and }y\leq{0}\}\cap\{(0,0)\}\nonumber\\
Q_{4} & := & \{(x,y):x\geq{0}\text{ and }y<0\}\cap\{(0,0)\}\nonumber
\end{eqnarray}
and switch when the solution trajectory enters a new quadrant.
We define
\begin{equation}
    \label{eqnVarStrucAp}
    \mathbf{\dot{x}}=\mathbf{f}_{p}(\mathbf{x})=A_{p}\mathbf{x}\text{, }p\in\{1,2\}
\end{equation}
where $p=1$ for $\mathbf{x}\in{Q_{2}\cup{Q_{4}}}$ and $p=2$ for $\mathbf{x}\in{Q_{1}\cup{Q_{3}}}$.\par
We can apply our algorithm to this example with $C_{1}=Q_{2}\cup{Q_{4}}$ and $C_{2}=Q_{1}\cup{Q_{3}}$, i.e.
$\Theta=2$, $P^{1}=\{1\}$ and $P^{2}=\{2\}$, with $\Omega^{P^{1}}=Q_{2}\cup{Q_{4}}$ and $\Omega^{P^{2}}=Q_{1}\cup{Q_{3}}$. Note that there is no intersection between our sets $C_{1}$ and $C_{2}$. The algorithm produces the plot for the Lyapunov function for this system as seen in the middle of Figure \ref{fig:Lyaps}, as well as plots for its orbital derivatives in Figure \ref{fig:VarStrucOrbDervs}, which are negative apart from the origin.
\begin{figure}[ht]
    \centering
    \includegraphics[width=0.45\linewidth]{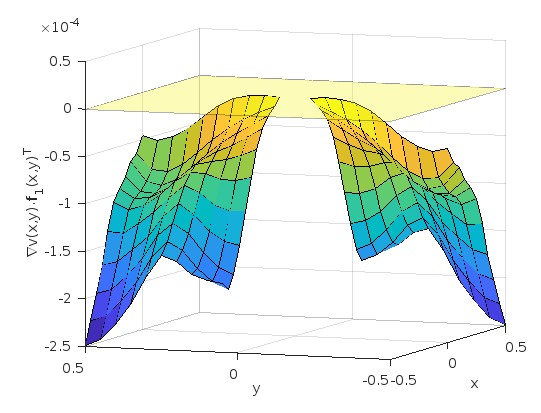}
    \includegraphics[width=0.45\linewidth]{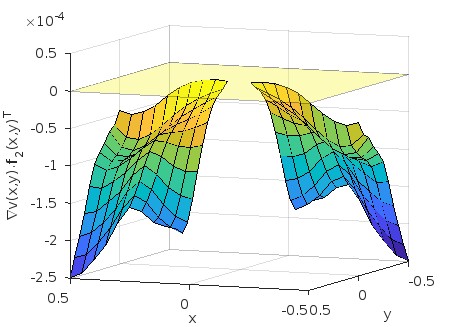}
    \caption{Plots of the orbital derivatives of the Lyapunov function displayed in the middle of Figure \ref{fig:Lyaps}: $\nabla{v}(x,y)\cdot\mathbf{f}_{1}(x,y)$ (left) and $\nabla{v}(x,y)\cdot\mathbf{f}_{2}(x,y)$ (right). The orbital derivatives are plotted in the quadrants where their respective function is defined. In both figures the plane $\{z=0\}$ has been plotted to illustrate that the constraints in (\ref{eqnMinProbv}) have been met.}
    \label{fig:VarStrucOrbDervs}
\end{figure}
\end{example}

\begin{example}{\rm (Arbitrary and Variable Structure system)}
    \label{exCombo}
    This example implements both time and state dependent switching. To do this we use the first system $\mathbf{f}_{1}$ from Example \ref{exArb}, and
    the two systems (\ref{eqnVarStrucAp}) from Example \ref{exVarStruct}. For ease of notation we label $\mathbf{f}_{2}(\mathbf{x}):=A_{1}\mathbf{x}$ and $\mathbf{f}_{3}(\mathbf{x}):=A_{2}\mathbf{x}$.\par
    Here we will be switching between $\mathbf{f}_{1}$ and $\mathbf{f}_{2}$ in the quadrants where $\mathbf{f}_{2}$ is defined, and between $\mathbf{f}_{1}$ and $\mathbf{f}_{3}$ in the quadrants where $\mathbf{f}_{3}$ is defined. These quadrants are specified in Example \ref{exVarStruct}. This can be described by the following system
    \begin{equation}
        \label{eqnComboSystem}
        \dot{\mathbf{x}}=\mathbf{f}_{p}(\mathbf{x})\text{, }p\in\{1,2,3\}
    \end{equation}
    where $p$ switches arbitrarily between $1$ and $2$ when $\mathbf{x}\in{Q_{2}\cup{Q_{4}}}$, and similarly between $1$ and $3$ when $\mathbf{x}\in{Q_{1}\cup{Q_{3}}}$.\par
    Our algorithm is applicable here with $C_{1}=\Omega$, $C_{2}=Q_{2}\cup{Q_{4}}$ and $C_{3}=Q_{1}\cup{Q_{3}}$. This then partitions our space with $\Theta=2$, $P^{1}=\{1,2\}$ and $P^{2}=\{1,3\}$, which means that $\Omega^{P^{1}}=Q_{2}\cup{Q_{4}}$ and $\Omega^{P^{2}}=Q_{1}\cup{Q_{3}}$. The algorithm produces the plot for the Lyapunov function for this system as seen on the right of Figure \ref{fig:Lyaps}, as well as plots for its orbital derivatives in Figure \ref{fig:ComboOrbDervs}. These final two figures show that
     the constraints in (\ref{eqnMinProbv}) have been met.
\begin{figure}[ht]
    \centering
    \includegraphics[width=0.45\linewidth]{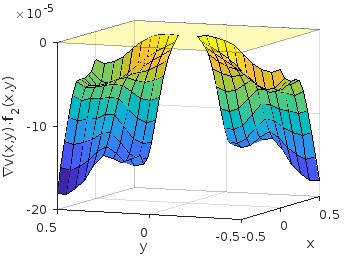}
    \includegraphics[width=0.45\linewidth]{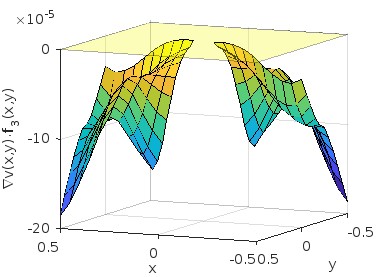}
    \includegraphics[width=0.45\linewidth]{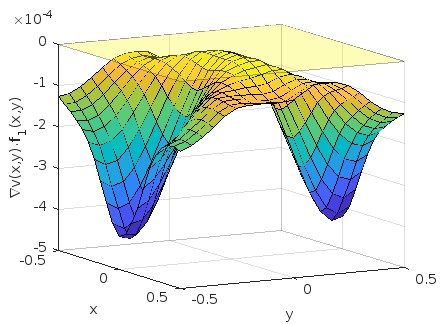}
    \caption{Plots of the orbital derivatives of the Lyapunov function displayed on the right of Figure \ref{fig:Lyaps}: $\nabla{v}(x,y)\cdot\mathbf{f}_{2}(x,y)$ (top left), $\nabla{v}(x,y)\cdot\mathbf{f}_{3}(x,y)$ (top right) and $\nabla{v}(x,y)\cdot\mathbf{f}_{1}(x,y)$ (bottom). For the systems $\mathbf{f}_{2}$ and $\mathbf{f}_{3}$ the orbital derivatives are plotted in the quadrants where their respective function is defined. The planes $\{z=0\}$ have been plotted to illustrate that the constraints in (\ref{eqnMinProbv}) have been met.}
    \label{fig:ComboOrbDervs}
\end{figure}
\end{example}

\section{Conclusion}

This paper presents a new method for the construction of a Lyapunov function for switched systems. The method is applicable to both arbitrary switched and variable structure systems, as well as systems that employ both time- and state-dependent switching. We have shown that the existence of a Lyapunov function implies that the origin is an asymptotically stable equilibrium point of the switched system. An algorithm to construct a Lyapunov function is presented which finds an approximate solution to a system of partial differential inequalities using meshfree collocation and quadratic programming.\par
Future work (\cite{JWConvergence}) will prove the existence of 
a Lyapunov function, and therefore the feasibility of problem (\ref{eqnbTAb}).
Moreover, we will study the convergence of the algorithm presented in this paper, i.e. we will show that the sequence of functions $v_n$ strongly converges to a solution $V$ of the original system of partial differential inequalities (\ref{eqnDV}). Here, $v_n$ is the solution of the discretised problem with $N_n$ collocation points, such that the appropriate fill distances converge to zero.
\medskip

\bibliographystyle{abbrv}
\bibliography{references}

\end{document}